\documentclass[12pt]{article}
\usepackage{amsthm,amsmath,amssymb,
extsizes,mathdots}

\newtheorem{theorem}{Theorem}
\newtheorem{lemma}[theorem]{Lemma}
\newtheorem{corollary}[theorem]{Corollary}
\theoremstyle{definition}
\newtheorem{remark}[theorem]{Remark}

\renewcommand{\le}{\leqslant}
\renewcommand{\ge}{\geqslant}

\DeclareMathOperator\diag{diag}
\DeclareMathOperator\Trace{trace}

\title{$\phantom{.}$\\[-30mm] Criterion of unitary similarity for upper triangular matrices in general position}

\author{Douglas Farenick\thanks{Department of Mathematics and Statistics,
University of Regina, Regina, Saskatchewan S4S 0A2, Canada. Email:
\mbox{Doug.Farenick@uregina.ca}.
Supported in part by NSERC.}
\and
Vyacheslav Futorny\thanks{Department of Mathematics, University of S\~ao Paulo, S\~ao Paulo, Brazil. Email: \mbox{futorny@ime.usp.br}. Supported  in part by the CNPq grant (301743/2007-0) and by the Fapesp grant (2010/50347-9).}
\and
Tatiana G.~Gerasimova\thanks{Faculty of Mechanics and Mathematics, Kiev National Taras Shevchenko University, Volodymyrska 64, Kiev, Ukraine. Emails: \mbox{gerasimova@imath.kiev.ua} (T.G.~Gerasimova) and \mbox{nadiia.shvai@gmail.com} (N.~Shvai).
The work was started while these authors were visiting the University of Regina supported by NSERC, the University of Regina Visiting Graduate Student Research Program, and the Svjatoslav Vakarchuk Foundation ``People of Future''.}
\and
Vladimir V. Sergeichuk\thanks{Institute of Mathematics, Tereshchenkivska 3, Kiev, Ukraine. Email: \mbox{sergeich@imath.kiev.ua}.
Supported in part by the Fapesp grant
(2010/07278-6).}
 \and
Nadya Shvai$^{\ddag}$}

\date{Dedicated to the Memory of Aleksandr Shvai, the husband of Nadya Shvai, who died tragically at age 23}

\begin{document}

\maketitle

\begin{abstract}
Each square complex matrix is
unitarily similar to an upper
triangular matrix with diagonal entries in any prescribed order. Let $A=[a_{ij}]$ and $B=[b_{ij}]$ be upper triangular $n\times n$ matrices that
\begin{itemize}
  \item are not similar to direct sums of matrices of smaller sizes, or
  \item are
in general position and have the same main diagonal.
\end{itemize}
We prove that
$A$ and $B$ are unitarily similar if
and only if
\[\|h(A_k)\|=\|h(B_k)\|\qquad \text{for all $h\in\mathbb
C[x]$  and $k=1,\dots,n$,}\]
where $A_k:=[a_{ij}]_{i,j=1}^k$ and $B_k:=[b_{ij}]_{i,j=1}^k$ are the
principal $k\times k$ submatrices
of $A$ and $B$ and $\|\cdot\|$ is
the Frobenius norm.

{\it MSC2010:}
15A21; 15A60

{\it Keywords:}
Unitary similarity; Classification; General position; Frobenius norm
\end{abstract}


\section{Introduction}\label{ore}

A classical problem in linear algebra is the following one: if $A$ and $B$ are square complex matrices, then how can one determine whether $A$ and $B$
are unitarily similar (i.e., $U^{-1}AU=B$ for a unitary $U$)? More precisely, which invariants completely determine a matrix up to unitary similarity?

Let us recall the most known solutions to this problem:
\begin{description}
  \item[Specht's theorem.] \emph{Matrices $A$ and $B$ are unitarily similar if and only if
\[
\Trace\, \omega(A,A^*)=
\Trace\, \omega(B,B^*)
\]
for all words $\omega$ in two noncommuting variables,} see \cite{spe}.

 \item[Littlewood's canonical matrices.] Littlewood \cite{lit} constructed an algorithm that reduces each square complex matrix $A$ by transformations of unitary similarity to some matrix $A_{\text{can}}$ in such a way that $A$ and $B$ are unitarily similar if and only if they are reduced to the same matrix $A_{\text{can}}=B_{\text{can}}$. Thus, the matrices that are not changed by Littlewood's algorithm are \emph{canonical  with respect to unitary similarity}. We use  Littlewood's canonical matrices in this paper (see Remark \ref{kme}).  Systems of linear mappings on unitary and Euclidean spaces (i.e., unitary and Euclidean representations of quivers) were studied in \cite{ser_unit} using Littlewood's algorithm.

 \item[Arveson's
criterion.] Let $A$ and $B$ be $n\times n$ complex matrices such that each of them is not unitarily similar to a direct sum of square matrices of smaller sizes. Arveson \cite[Theorems 2 and 3]{arv0} proved that \emph{$A$ and $B$ are unitarily similar if and only if
\begin{equation}\label{feu}
\|H_0\otimes I_n+H_1\otimes A\|_{op}=\|H_0\otimes I_n+H_1\otimes B\|_{op}
\end{equation}
for all $H_0,H_1\in \mathbb C^{n\times n}$, where $\|M\|_{op}:=\max_{|v|=1}|Mv|$ is the operator norm and $|\cdot|$
stands for the Euclidean norm of vectors.}

For each matrix polynomial  \[H(x) = H_0  +H_1x+\dots+  H_tx^t\in\mathbb C^{k\times k}[x],\] whose coefficients $H_i$ are $k\times k$  matrices, we define its value at an $n\times n$ matrix $M$ as follows:
\[H(M):=H_0\otimes I_n  +H_1\otimes M+\dots+  H_t\otimes M^t\in \mathbb C^{kn\times kn}.\]
The condition \eqref{feu} means that
\begin{equation}\label{dwl}
\|H(A)\|_{op}=\|H( B)\|_{op}
\end{equation}
for all matrix polynomials
$H\in\mathbb C^{n\times n}[x]$ of degree at most $1$.
For some class of operators on a Hilbert space, Arveson \cite[Theorem 2.3.2]{arv2} proved that two operators $A$ and $B$ are unitarily similar if and only if the condition \eqref{dwl} holds for all (possibly, nonlinear)  $H\in\mathbb C^{k\times k}[x]$.
\end{description}

The purpose of this paper is to give a criterion of unitary similarity of matrices that is analogous to Arveson's criterion \eqref{dwl}, but in which polynomials over $\mathbb C$ are used instead of linear polynomials over $\mathbb C^{n\times n}$. All matrices that we consider are complex matrices.

We study only the finite dimensional case, and so we can and will use the \emph{Frobenius norm}
\begin{equation*}\label{epk}
\|A\|:=\sqrt{\sum
|a_{ij}|^2}, \qquad\text{where } A=[a_{ij}]\in \mathbb C^{n\times n},
\end{equation*}
instead of the operator norm.
The \emph{Frobenius norm of a linear operator} on a unitary space is the Frobenius norm of its matrix in any orthonormal basis. This definition is correct since the Frobenius norm of a matrix does
not change under multiplication by unitary matrices. Hence, if $A$ and $B$ are unitarily similar matrices, then $\|A\|=\|B\|$; moreover,
\begin{equation}\label{blr}
\|h(A)\|=\|h( B)\|\qquad \text{for all
$h\in\mathbb C[x]$}.
\end{equation}

The converse statement is not true;
the condition \eqref{blr}
does not ensure the unitary similarity of matrices:
\begin{equation}\label{usp}
A=\begin{bmatrix}
    0&1&0\\0&0&2\\0&0&0
  \end{bmatrix},\qquad B=\begin{bmatrix}
    0&2&0\\0&0&1\\0&0&0
  \end{bmatrix}
\end{equation}
are not unitarily similar and
satisfy \eqref{blr}; see Lemma \ref{krd}. But their $2\times 2$ principal submatrices
\[
A_2=\begin{bmatrix}
    0&1\\0&0
  \end{bmatrix},\qquad B_2=\begin{bmatrix}
    0&2\\0&0
  \end{bmatrix}
\]
do not satisfy \eqref{blr}. (By the $k\times k$ \emph{principal submatrix} $M_k$ of a matrix $M$, we mean the submatrix at the intersection of the first $k$ rows and the first $k$ columns.)
For this reason, we give a criterion of unitary similarity, in which the condition \eqref{blr} is imposed not only on $n\times n$ matrices $A$ and $B$, but also on their principal submatrices:
\begin{equation}\label{blh6}
\|h(A_k)\|=\|h(B_k)\|\qquad
\text{for all }h\in\mathbb
C[x]\text{ and }k=1,\dots,n.
\end{equation}

We prove that the condition \eqref{blh6} ensures the unitary similarity of \emph{upper triangular} $n\times n$ matrices $A$ and $B$ in two cases:
\begin{itemize}
  \item if $A$ and $B$ are not similar to direct sums of square matrices of smaller sizes (Theorem \ref{jom}), and

  \item if $A$ and $B$ are in general position (Theorem \ref{jir}).
\end{itemize}

We consider only upper triangular matrices because of the Schur unitary triangularization theorem \cite[Theorem 2.3.1]{HJ1}:
\emph{every square matrix $A$ is unitarily similar to an upper triangular matrix $B$ whose diagonal entries are complex numbers in any prescribed order}; say, in the lexicographical order:
\begin{equation}\label{gep}
\text{$a+bi\preccurlyeq c+di$\qquad if either $a<c$, or $a=c$ and $b\le d$.}
\end{equation}
A unitary matrix $U$ that transforms $A$ to $B=U^{-1}AU$  is easily constructed: we reduce $A$ by similarity transformations to an upper triangular matrix $S^{-1}AS$ with diagonal entries in the prescribed order (this matrix can be obtained from the Jordan form of $A$ by simultaneous permutations of rows and columns), then apply the Gram-Schmidt orthogonalization to the columns of $S$ and obtain a desired unitary matrix $U=ST$, where $T$ is upper triangular.

\section{Main results}\label{ooe}

\subsection{Criterion for indecomposable matrices and unicellular operators}

We say that a matrix is \emph{indecomposable for similarity} if it is not similar to a direct sum of square matrices of smaller sizes. This means that the matrix is similar to a Jordan block. Thus, a matrix is indecomposable with respect to similarity if and only if it is unitarily similar to a matrix of the form
\begin{equation}\label{jys1}
A=\begin{bmatrix}
     \lambda &a_{12}&\dots&a_{1n} \\
     &\lambda&\ddots&\vdots \\
     &&\ddots&a_{n-1,n}\\ 0&&&\lambda
   \end{bmatrix},\qquad \text{all }
     a_{i,i+1}\ne 0.
\end{equation}

In Section \ref{s2} we prove the following theorem, which is the first main result of the paper.

\begin{theorem}\label{jom}
Let $A$ and $B$ be $n\times n$ upper triangular matrices that
are indecomposable with respect to similarity. Then $A$ and $B$ are
unitarily similar if and only if
\begin{equation}\label{blh1}
\|h(A_k)\|=\|h(B_k)\|\qquad
\text{for all }h\in\mathbb
C[x]\text{ and }k=1,\dots,n,
\end{equation}
where $A_k$ and $B_k$ are the
principal $k\times k$ submatrices
of $A$ and $B$.
\end{theorem}

Now we give the operator form of this criterion. Two operators $\cal A$ and $\cal B$ on a unitary space are \emph{unitarily similar} if there exists a unitary operator ${\cal U}$ such that ${\cal U}^{-1}{\cal A}{\cal U}={\cal B}$.
A linear operator ${\cal A}:U\to U$ on an $n$-dimensional unitary space $U$ is said to be \emph{unicellular} if it satisfies one of the following 3 equivalent conditions:
\begin{itemize}
  \item its matrix is indecomposable with respect to similarity;
  \item there exist no invariant subspaces $U'$ and $U''$ of $\cal A$ such that \[\dim U'+\dim U''=n,\qquad U'\cap U''=0;\]
  \item
all invariant subspaces of $\cal A$ form the chain
\begin{equation*}\label{ees}
0 \subset U_1 \subset U_2 \subset \ldots \subset U_n=U,\qquad \dim U_i=i\text{ for all }i.
\end{equation*}
\end{itemize}

\begin{corollary}\label{ttt1}
{\rm(a)} Let $\cal A$ and $\cal B$ be unicellular linear operators on an $n$-dimensional unitary space $U$ with the chains of invariant subspaces
\[
0 \subset U_1 \subset U_2 \subset \ldots \subset U_n=U,\qquad 0\subset V_1\subset V_2\subset\ldots\subset V_n=U,
\]
and let \[{\cal A}_i:={\cal A}|U_k,\qquad {\cal B}_k:={\cal B}|V_k
\] be the
restrictions of $\cal A$ and $\cal B$ to their invariant subspaces.
Then $\cal A$ and $\cal B$
are unitarily similar if and only if
\begin{equation}\label{bld}
\|h({\cal
A}_k)\|=\|h({\cal B}_k)\|\qquad
\text{for all }h\in\mathbb
C[x]\text{ and }k=1,\dots,n.
\end{equation}

{\rm(b)} In particular, two nilpotent linear operators $\cal A$ and ${\cal B}$ of rank $n-1$ on an $n$-dimensional unitary space are unitarily similar if and only if \eqref{bld} holds for the restrictions ${\cal A}_k$ and ${\cal B}_k$ of the operators to the images of ${\cal A}^k$ and ${\cal B}^k$.
\end{corollary}

Let \eqref{bld} hold. Then $\cal A$ and $\cal B$ have the same eigenvalue: if $\lambda$ is the eigenvalue of $\cal A$ and $h(x):=(x-\lambda)^n$, then $\|h({\cal
B})\|=\|h({\cal A})\|=0$ and so $\lambda $ is the eigenvalue of $\cal B$. Hence, the canonical
isomorphism of one-generated
algebras
\begin{equation}\label{hkt}
{\mathbb C}[{\cal A}_k]\simeq
{\mathbb C}[{\cal B}_k],\qquad
{\cal A}_k\mapsto {\cal B}_k
\end{equation}
is defined correctly: the algebras are isomorphic to ${\mathbb C}[x]/(x-\lambda)^k{\mathbb C}[x]$.

\begin{corollary}\label{qj2}
Two unicellular linear operators $\cal A$ and $\cal B$ on an $n$-dimensional unitary space are unitarily similar if and only if they have the same eigenvalue and the canonical isomorphism \eqref{hkt} is isometric $($i.e., it preserves the norm$)$ for each
$k=1,\dots,n$.
\end{corollary}

\subsection{Criterion for matrices in general position}

Theorem \ref{jom} is not extended to matrices with several eigenvalues:
we prove in Lemma \ref{lor} that each two matrices of the form \begin{equation}\label{dft}
A:=\begin{bmatrix}
  0&1&-1&a\\0&1&1&1\\0&0&2&1\\0&0&0&3
\end{bmatrix},\qquad B:=\begin{bmatrix}
  0&1&-1&b\\0&1&1&1\\0&0&2&1\\0&0&0&3
\end{bmatrix},\quad \begin{matrix}a\ne b,\\ |a|=|b|=1,
\end{matrix}
\end{equation}
are not unitarily similar but satisfy \eqref{blh1}.
Nevertheless, in this section we extend Theorem \ref{jom} to ``almost all'' upper triangular matrices as follows.

Let \begin{equation}\label{gree}
X_n:=\begin{bmatrix}
     x_{11} &\dots&x_{1n} \\
     &\ddots&\vdots \\ 0&&x_{nn}
   \end{bmatrix}
\end{equation}
be a matrix whose upper triangular entries are variables;
denote by $\mathbb C[x_{ij}|i\le j\le n]$ the set of
polynomials in these variables.
For simplicity of notation, we write $f\{X_n\}$ instead of $f(x_{11},x_{12},x_{22},\dots)$.

For each $f\in
\mathbb C[x_{ij}|i\le j\le
n]$, write
\begin{equation}\label{wmo}
M_n(f):=\{A\in \mathbb C^{n\times n}\,|\, A\text{ is upper triangular and }f\{A\}\ne 0\}.
\end{equation}

For example, if
\begin{equation}\label{kut}
\varphi_n\{X_n\}:=x_{12}x_{23}\cdots
x_{n-1,n}\prod_{i<
j}(x_{ii}-x_{jj}),
\end{equation}
then $M_n(\varphi_n)$ consists of
matrices of the form \begin{equation}\label{jys}
\begin{bmatrix}
     \lambda_{1} &a_{12}&\dots&a_{1n} \\
     &\lambda _2&\ddots&\vdots \\
     &&\ddots&a_{n-1,n}\\ 0&&&\lambda _{n}
   \end{bmatrix},\qquad \begin{matrix}
     \lambda _i\ne \lambda _j\text{ if }i\ne j,\\
     \text{all }a_{i,i+1}\ne 0.
   \end{matrix}
\end{equation}

We say that $n \times n$ upper triangular matrices \emph{in general position} possess some property
if there exists a nonzero polynomial $f_n\in
\mathbb C[x_{ij}|i\le j\le
n]$ such that all matrices in $M_n(f_n)$ possess this property. Thus, this property holds for all matrices in $\mathbb C^{n\times n}$ except for matrices from an algebraic variety of smaller dimension.\footnote{In algebraic
geometry, when a family of objects
$\{X_p\}_{p\in\Sigma}$ is
parametrized by the points of an
irreducible algebraic variety
$\Sigma$, the statement that ``the
general object $X$ has a property
$\cal P$" is taken to mean that
``the subset of points
$p\in\Sigma$ such that the
corresponding object $X_p$ has the
property $\cal P$ contains a
Zariski open dense subset of
$\Sigma$'', see \cite[p.\,54]{har}.}

The second main result of the paper is the following theorem.

\begin{theorem}\label{jir}
Two $n\times n$ upper triangular matrices $A$ and $B$ in general position with lexicographically ordered eigenvalues on the main diagonal $($see \eqref{gep}$)$ are
unitarily similar if and only if \begin{equation}\label{blh}
\|h(A_k)\|=\|h(B_k)\|\qquad
\text{for all }h\in\mathbb
C[x]\text{ and }k=1,\dots,n,
\end{equation}
where $A_k$ and $B_k$ are the
principal $k\times k$ submatrices
of $A$ and $B$.
\end{theorem}

Theorem \ref{jir} is an existence theorem: ``$A$ and $B$ in general position'' means ``$A,B\in M_n(f_n)$ for some $f_n$''. In Theorem \ref{thms}, we give $f_n$ in an explicit form.

For each $n\ge 2$ and $r=1,2,\dots,n$, define the $n\times n$ matrix
\begin{multline}\label{zzz}
G^{(n,r)}\{X_n\}=[g_{ij}^{(n,r)}\{X_n\}]\\
:=
  \begin{cases}
    (X_n-x_{22}I_n)(X_n-x_{33}I_n)\cdots
 (X_n-x_{nn}I_n) & \text{if }r=1, \\
    (X_n-x_{11}I_n)(X_n-x_{22}I_n)
 \cdots
 (X_n-x_{r-1,r-1}I_n) & \text{if }r>1.
  \end{cases}
\end{multline}
Its entries $g_{ij}^{(n,r)}\{X_n\}$ are polynomials in entries of \eqref{gree}.
Write
\begin{equation}\label{gun}
f_n:=
  \begin{cases}
    \varphi_n & \text{if $n=1,2,3$}, \\
\varphi_n
\cdot  g_{14}^{(4,1)}g_{15}^{(5,1)}\cdots g_{1n}^{(n,1)}
\cdot g_{13}^{(3,3)}g_{14}^{(4,4)}\cdots g_{1,n-1}^{(n-1,n-1)}
 & \text{if $n\ge 4$}.
  \end{cases}
\end{equation}
in which $\varphi_n$ is defined in \eqref{kut}.
Theorem \ref{jir} results from the following theorem, which is proved in Section \ref{s3}.

\begin{theorem}\label{thms}
Matrices $A,B\in M_n(f_n)$ are unitarily similar and have the same main diagonal if and only if
\eqref{blh} holds.
\end{theorem}

By this theorem and the top equality in \eqref{gun}, \emph{two matrices $A$ and $B$ of the form \eqref{jys} of size at most $3\times 3$ are unitarily similar if and only if
\eqref{blh} holds.}

\section{Proof of Theorem \ref{jom}}
\label{s2}

\begin{lemma}\label{oiy}
{\rm(a)} For each
\begin{equation}\label{jysr}
A=\begin{bmatrix}
     \lambda_{1} &a_{12}&\dots&a_{1n} \\
     &\lambda _2&\ddots&\vdots \\
     &&\ddots&a_{n-1,n}\\ 0&&&\lambda _{n}
   \end{bmatrix},\qquad \text{all }a_{i,i+1}\ne 0,
\end{equation}
there exists a diagonal unitary
matrix $U$ such that all the entries of the first superdiagonal of $U^{-1}AU$ are positive real numbers.

{\rm(b)} Let
\begin{equation}\label{swJn}
A=\begin{bmatrix}
     \lambda_1 &a_{12}&\dots&a_{1n} \\
     &\lambda_2&\ddots&\vdots \\
     &&\ddots&a_{n-1,n}\\ 0&&&\lambda_n
   \end{bmatrix},\qquad \text{all $a_{i,i+1}$ are positive real,}
\end{equation}
and
\[
B=\begin{bmatrix}
     \lambda_{1} &b_{12}&\dots&b_{1n} \\
     &\lambda _2&\ddots&\vdots \\
     &&\ddots&b_{n-1,n}\\ 0&&&\lambda _{n}
   \end{bmatrix},\qquad \text{all $b_{i,i+1}$ are positive real.}
\]
If $A$ and $B$ are unitarily similar, then $A=B$. Moreover, if $U^{-1}AU=B$ and $U$ is a unitary matrix, then $U=uI_n$ for some $u\in \mathbb C$ with $|u|=1$.
\end{lemma}

\begin{proof}
(a) Write $a_{i,i+1}$ in the form $r_{i}u_{i}$, in which $r_{i}$ is a positive real number and $|u_{i}|=1$. Then
\[
U^{-1}:=\diag(1, u_{1},  u_{1} u_{2},u_{1} u_{2}u_3,\dots)
\]
is the desired matrix.

(b) Let $U^{-1}AU=B$, in which $U$ is a unitary matrix. Equating the entries of $AU=UB$ along diagonals starting at the lower left diagonal (i.e., from the entry $(n,1)$) and finishing at the main diagonal, we find that $U$ is upper triangular. Since $U$ is unitary, it is a diagonal matrix: $U=\diag(u_1,\dots,u_n)$. Equating the entries of $AU=UB$ along the first superdiagonal, we find that $u_1=\dots=u_n$. Hence, $U=uI_n$ and $A=B$.
\end{proof}

\begin{remark}\label{kme} By Lemma \ref{oiy}(b), any two matrices of the form \eqref{swJn} in which the diagonal entries are lexicographically ordered (i.e.,
$\lambda _1\preccurlyeq\dots
     \preccurlyeq\lambda_n,$ see \eqref{gep})
are either equal or unitarily dissimilar. These matrices are Littlewood's canonical forms of matrices \eqref{jysr}; see the beginning of Section \ref{ore}.
\end{remark}

\begin{lemma}\label{ajs}
Each matrix of the form
\begin{equation}\label{swJ}
A=\begin{bmatrix}
     \lambda &a_{12}&\dots&a_{1n} \\
     &\lambda&\ddots&\vdots \\
     &&\ddots&a_{n-1,n}\\ 0&&&\lambda
   \end{bmatrix},\qquad \text{all $a_{i,i+1}$ are positive real,}
\end{equation}
is fully determined by the indexed family of real numbers
\begin{equation}\label{gdm}
 \{\|h({A}_k)\|\}_{(h,k)},\qquad
\text{in which $h\in\mathbb C[x]$ and
$k=1,\dots,n.$}
\end{equation}
\end{lemma}

\begin{proof}
Let $h$ be a nonzero polynomial of minimal degree such that $\|h(A)\|=0$. Then $h(A)=0$ and so $h(A)=(x-\lambda )^n$. Thus, $\lambda $ is determined by \eqref{gdm}.

Write $B:=A-\lambda I_n$. Then \eqref{gdm} determines the family
\begin{equation}\label{gdw}
 \{\|h({B}_k)\|\}_{(h,k)},\qquad
\text{in which $h\in\mathbb C[x]$ and
$k=1,\dots,n.$}
\end{equation}
The positive real number $a_{12}$ is determined by \eqref{gdw} since
$\|B_2\|= a_{12}$.
This proves the lemma for $n=1$ and $2$.

Reasoning by induction on $n$, we assume that $n\ge 3$ and $B_{n-1}$ is determined by \eqref{gdw}.
Since all entries of $B^{n-1}$ are zero except for the $(1,n)$ entry, which is the positive real number
\begin{equation*}\label{gtk}
c:=a_{12}a_{23}\cdots a_{n-1,n},
\end{equation*}
we have $\|B^{n-1}\|=c$. Thus, $a_{n-1,n}$ is determined by \eqref{gdw}.

Reasoning by induction, we assume
that $a_{n-1,n},a_{n-2,n},\dots,
a_{r+1,n}$ are determined by
\eqref{gdw} and find $a_{rn}$.
Let $\alpha $ be a complex number for which $\|B^{r }-\alpha B^{n-1}\|$ is minimal. Then the $(1,n)$ entry of $B^{r }-\alpha B^{n-1}$
is
\[
a_{12}a_{23}\cdots a_{r-1,r}a_{rn}+\dots+\alpha c=0.
\]
Since the unspecified summands  do not contain $a_{rn}$ and only $a_{rn}$ is unknown in this equality, it determines $a_{rn}$.
\end{proof}

\begin{proof}[{Proof of Theorem
\ref{jom}}] \label{kay}
Let $M$ be an $n\times n$ upper triangular matrix that
is indecomposable with respect to similarity. By Lemma \ref{oiy}(a), $M$ is unitarily similar to a
matrix $A$ of the form \eqref{swJ} via a diagonal unitary matrix.
Then $\|h(M_k)\|=\|h(A_k)\|$ for all
$h\in\mathbb C[x]$ and
$k=1,\dots,n$. Thus, it suffices to prove Theorem
\ref{jom} for matrices of the form  \eqref{swJ}.

``$\Rightarrow$'' Let $A$ and $B$ of the form \eqref{swJ} be unitarily similar. By Lemma \ref{oiy}(b), $A=B$, and so \eqref{blh1} holds.

``$\Leftarrow$'' Let $A$ and $B$ of the form \eqref{swJ} satisfy \eqref{blh1}. By Lemma \ref{ajs}, $A=B$ since their indexed families  $\{\|h({A}_k)\|\}_{(h,k)}$ and
$\{\|h({B}_k)\|\}_{(h,k)}$ coincide.
\end{proof}

\section{Counterexamples}

\subsection{Condition \eqref{blr} does not ensure the unitary similarity}\label{ajt}

In this section, we give examples of matrices of the form \eqref{swJn} (and even of the form \eqref{swJ}), for which the condition \eqref{blr}
does not ensure their unitary similarity.

For each square matrix $A$, denote by $A^S$ its transpose with respect to the secondary diagonal:
\begin{equation*}\label{hhr}
A^S=ZA^TZ,\qquad Z:=\begin{bmatrix}
                    0 && 1 \\
                    &\iddots&\\
                    1 && 0 \\
                  \end{bmatrix}.
\end{equation*}
For instance, $B =A^S$ in \eqref{usp}.

\begin{lemma}\label{krd}
Let $A$ be a matrix of the form \eqref{swJn} such that $A\ne A^S$ and the main diagonals of $A$ and $A^S$ coincide $($i.e., the main diagonal of $A$ is symmetric$)$. Then $A$ and $B:=A^S$ satisfy \eqref{blr}, but they are not unitarily similar.
\end{lemma}

\begin{proof}
The condition \eqref{blr} holds for $A$ and $B=A^S$, because $
\|h(A^S)\|= \|h(ZA^TZ)\|= \|Zh(A^T)Z\| =\|h(A^T)\|=\|h(A)^T\|=\|h(A)\|.
$

Since $A\ne A^S$, $A$ and $A^S$ are not unitarily similar
by Lemma \ref{oiy}(b).
\end{proof}

\subsection{Theorem \ref{jom} is not extended to matrices with several eigenvalues}

Theorem \ref{jom} was proved for matrices of the form \eqref{jys1}; let us show that it is not extended to matrices of the form \eqref{jys}.

\begin{lemma}\label{lor}
Matrices $A$ and $B$ of the form
\eqref{dft}
are not unitarily similar but satisfy \eqref{blh1}.
\end{lemma}

\begin{proof}
By Lemma \ref{oiy}(b), $A$ and $B$ are not unitarily similar since $a\ne b$.

Let us prove, that $A$ and $B$ satisfy \eqref{blh1}.
Write
\[
M_c:=\begin{bmatrix}
  0&1&-1&c\\0&1&1&1\\0&0&2&1\\0&0&0&3
\end{bmatrix},\qquad\text{in which $c\in\mathbb C$ and $|c|=1$},
\]
and take any $h\in\mathbb
C[x]$.

It suffices to
prove that $\|h(M_c)\|$ does not depend on $c$.
Let $r(x)= \alpha +\beta x+\gamma
x^2+\delta x^3$ be the residue of division of $h(x)$ by the characteristic
polynomial of $M_c$. Then
\[
h(M_c)=r(M_c)=\alpha I_4+\beta M_c+
\gamma
\begin{bmatrix}
  0&1&-1&3c\\0&1&3&5\\0&0&4&5\\0&0&0&9
\end{bmatrix}+\delta
\begin{bmatrix}
  0&1&-1&9c\\0&1&7&19\\0&0&8&19
  \\0&0&0&27
\end{bmatrix},
\]
and so
\begin{align*}
\|h(M_c)\|^2=\|r(M_c)\|^2=&\|r(M_0)\|^2+ |\beta
c+\gamma 3c+\delta 9c|^2\\
=&\|r(M_0)\|^2+ |\beta +3\gamma
+9\delta|^2|c|^2\\=&\|r(M_0)\|^2+ |\beta +3\gamma
+9\delta|^2
\end{align*}
does not depend on $c$.
\end{proof}

Note that $M_c\notin M_4(f_4)$ (in which $M_4(f_4)$ from Theorem \ref{thms})  since $g_{13}^{(3,3)}\{M_c\}=0$.

\section{Proof of Theorem \ref{thms}}
\label{s3}

In this section, $M_n(f)$ is the set \eqref{wmo} and $f_n$ is the polynomial \eqref{gun}.

\begin{lemma}\label{lhs} Let $G^{(n,r)}$ be the matrix defined in \eqref{zzz}.

{\rm(a)} Only the first row of $G^{(n,1)}$ is nonzero.

{\rm(b)} The matrix $G^{(n,r)}$ with $2\le r\le n$
has the form
\begin{equation}\label{hiw}
\begin{bmatrix}
  0_{r-1} & * \\
  0 & T \\
\end{bmatrix},
\end{equation}
in which $0_{r-1}$ is the $(r-1)\times (r-1)$ zero matrix and $T$ is upper triangular.

{\rm(c)} The matrix $G^{(r,r)}$ with
$2\le r< n$ is the $r\times r$ principal submatrix of $G^{(n,r)}$.
\end{lemma}

\begin{proof} For every $i=1,\dots,n$,
let $P_i$ be any $n\times n$ upper triangular matrix, in which the $(i,i)$ entry is zero. Then
\begin{equation}\label{gpk}
P_1\cdots P_n=\begin{bmatrix}
  0 &&&* \\
   & *&& \\
   & &\ddots& \\
   0&&&*
\end{bmatrix}
\begin{bmatrix}
  * &&&* \\
   & 0&& \\
   &&*&\\
   0& &&\ddots
\end{bmatrix}\cdots
\begin{bmatrix}
  * &&&*  \\
   & \ddots&& \\
   &&*&\\
   0& &&0
\end{bmatrix}
=0.
\end{equation}
This equality is proved by induction on $n$: if it holds for $n-1$, then the product of the $(n-1)\times (n-1)$ principal submatrices of $P_1,\dots,P_{n-1}$ is zero, and so
\[
(P_1\cdots P_{n-1})P_n=
\begin{bmatrix}
  0 &\dots&0&* \\
   & \ddots&\vdots&\vdots \\
   &&0&*\\0&&&*
\end{bmatrix}
\begin{bmatrix}
  * &&&*  \\
   & \ddots&& \\
   &&*&\\
   0& &&0
\end{bmatrix}=0.
\]

(a) In the product of matrices \eqref{zzz} that defines $G^{(n,1)}$, we remove the first row and the first column in each of its factors. Then apply \eqref{gpk} to the obtained product.

(b) In the product of matrices \eqref{zzz} that defines $G^{(n,r)}$ with $r\ge 2$, we replace each factor by its $(r-1)\times(r-1)$ principal submatrix. Then apply \eqref{gpk} to the obtained product.

(c) This statement follows from \eqref{zzz}.
\end{proof}

\begin{lemma}\label{ajs2}
If $A\in M_n(f_n)$ and $S$ is a nonsingular diagonal matrix, then $S^{-1}AS\in M_n(f_n)$.
\end{lemma}

\begin{proof}
Let $A\in M_n(f_n)$ and let $S$ be a nonsingular diagonal matrix.
For each $i$, the $(i,i)$ entries of $A$ and $S^{-1}AS$ coincide and $S^{-1}AS-a_{ii}I_n=S^{-1}(A-a_{ii}I_n)S$. Thus, $G^{(n,r)}\{S^{-1}AS\}= S^{-1}G^{(n,r)}\{A\}S$ for each $r$, and so the corresponding entries of $G^{(n,r)}\{A\}$ and $G^{(n,r)}\{S^{-1}AS\}$ are simultaneously zero or nonzero. Taking into account the definition \eqref{gun} of $f_n$, we get $S^{-1}AS\in M_n(f_n)$.
\end{proof}

The following lemma is analogous to Lemma \ref{ajs}.

\begin{lemma}\label{ajs1}
Each matrix $A\in M_n(f_n)$, in which all entries of the first superdiagonal are positive real numbers,
is fully determined by the indexed family of real numbers
\begin{equation}\label{gdm1}
 \{\|h({A}_k)\|\}_{(h,k)},\qquad
\text{in which $h\in\mathbb C[x]$ and
$k=1,\dots,n.$}
\end{equation}
\end{lemma}

\begin{proof}
The matrix $A$ has the form
\begin{equation}\label{jys8}
\begin{bmatrix}
     \lambda_{1} &a_{12}&\dots&a_{1n} \\
     &\lambda _2&\ddots&\vdots \\
     &&\ddots&a_{n-1,n}\\ 0&&&\lambda _{n}
   \end{bmatrix},\qquad \begin{matrix}
     \lambda _i\ne \lambda _j\text{ if }i\ne j,\\
     \text{all $a_{i,i+1}$ are positive real.}
   \end{matrix}
\end{equation}

For each $k$, the minimal polynomial $\mu _k(x)$ of
$A_k$  is determined by the family \eqref{gdm1}. Since
$\lambda _i\ne \lambda _j$ if
$i\ne j,$ $\mu _k(x)$ is the
characteristic polynomial of $A_k$, and so $\mu _k(x)/\mu
_{k-1}(x)=x- \lambda _k$. Thus, the main
diagonal of $A$ is determined by
\eqref{gdm1}.

The entry $a_{12}$ of
$A$ is also determined by
\eqref{gdm1} since $a_{12}$ is a positive real
number, $\|A_2\|$ is determined by
\eqref{gdm1}, and
\[\|A_2\|^2=|\lambda_1|^2
+|\lambda_2|^2+a_{12}^2.\]
This proves Lemma \ref{ajs1} for $n=2$.

Let $n\ge 3$. Since $f_{n-1}$ divides $f_n$, we can use induction on $n$, and so we
assume that \eqref{gdm1} determines  $A_{n-1}$.

Let us find $a_{n-1,n}$. For each $r=2,3,\dots,n-1$,
define the $n\times n$ matrix
\begin{equation*}\label{sym}
B^{(r)}=[b_{ij}^{(r)}]:=G^{(n,r)}\{A\}(A-\lambda _nI_n)=(A-\lambda _1I_n)
 \cdots
 (A-\lambda _{r-1}I_n)(A-\lambda _nI_n).
\end{equation*}
Since $G^{(n,n-1)}\{A\}$ has the form \eqref{hiw} and the $(n,n)$ entry of $A-\lambda _nI_n$ is zero, the last column of $B^{(n-1)}$ is $va_{n-1,n}$, in which
\begin{equation}\label{fsy}
v:=[g_{1,n-1}^{(n,n-1)}\{A\},\dots,
g_{n-1,n-1}^{(n,n-1)}\{A\},0]^T
\end{equation}
is the $(n-1)$th column of $G^{(n,n-1)}\{A\}$. The column $v$ is known since it is determined by
$A_{n-1}$. The first coordinate of $v$ is nonzero since by Lemma \ref{lhs}(c)
\begin{equation}\label{cws}
g_{1r}^{(n,r)}\{A\}= g_{1r}^{(r,r)}\{A\}\ne 0,\qquad r=2,\dots, n-1;
\end{equation}
these are nonzero since each $g_{1r}^{(r,r)}$ divides $f_n$ (note that $g_{12}^{(2,2)}=x_{12}$).

Thus, $\|v\|\ne 0$.
The positive real
number
$a_{n-1,n}$ is fully determined by
the equality
\[
\|B^{(n-1)}\|^2=
\|B_{n-1}^{(n-1)}\|^2
+\|v\|^2 a_{n-1,n}^2,
\]
in which $B_{n-1}^{(n-1)}$ is the
$(n-1)$-by-$(n-1)$ principal
submatrix of $B^{(n-1)}$.

We have
determined the matrix $B^{(n-1)}$ too since
its last column is $va_{n-1,n}$.

Let us
consider the space $\mathbb C^{n\times n}$ of $n$-by-$n$
matrices as the unitary space with scalar product
\[
(X,Y):=\sum_{i,j} x_{ij}\bar
y_{ij},\qquad X=[x_{ij}],\ Y=[y_{ij}]\in \mathbb C^{n\times n}.
\]
This scalar product is expressed
via the Frobenius norm due to the
polarization identity
\begin{equation*}\label{fel}
(X,Y)=\frac14(\|X+Y\|^2-\|X-Y\|^2)
+\frac i4(\|X+iY\|^2- \|X-iY\|^2).
\end{equation*}

By Lemma \ref{lhs}(a),
\begin{equation*}\label{lge}
C:=G^{(n,1)}\{A\}= (A-\lambda _2I_n)\cdots
 (A-\lambda _nI_n)=\begin{bmatrix}
     c_{1} &c_{2}&\dots&c_{n} \\
     &0&\dots&0 \\
     &&\ddots&\vdots\\ 0&&&0
   \end{bmatrix},
\end{equation*}
in which $c_1,\dots,c_{n-1}$ are known. Since the main diagonal of $A$ is determined by \eqref{gdm1}, $\|C\|$ is determined by \eqref{gdm1} too. Using the polarization identity, we find $(B^{(n-1)},C)$. Then we find $c_n$ from the equality
\begin{equation*}\label{rzo}
(B^{(n-1)},C) =b^{(n-1)}_{11}\bar
c_1+ \cdots +b^{(n-1)}_{1,n-1}\bar
c_{n-1}+b^{(n-1)}_{1n}\bar c_n,
\end{equation*}
in which
$b^{(n-1)}_{1n}=g_{1,n-1}^{(n,n-1)}\{A\}a_{n-1,n}
\ne 0$ (see \eqref{fsy} and \eqref{cws}).

Reasoning by induction, we assume
that $a_{n-1,n},a_{n-2,n},\dots,
a_{r+1,n}$ are known and find $a_{rn}$ for each $r\le n-2$.

Suppose first that $r\ge 2$. Then $n\ge 4$, and $c_n\ne 0$ because $g_{1n}^{(n,1)}$ divides $f_n$.
Since $\|B^{(r)}\|$ is determined by \eqref{gdm1} and $C$ is known, we determine $(B^{(r)},C)$ using the
polarization identity. We determine
$b^{(r)}_{1n}$
from
\begin{equation*}\label{rzo2}
(B^{(r)},C) =b^{(r)}_{11}\bar c_1+
\cdots +b^{(r)}_{1,n-1}\bar
c_{n-1}+b^{(r)}_{1n}\bar c_n.
\end{equation*}
By \eqref{hiw}, the first $r-1$ columns of $G^{(n,r)}$ are zero, and so
\[b^{(r)}_{1n}
=g_{1r}^{(n,r)}\{A\}a_{rn}
+g_{1,r+1}^{(n,r)}\{A\}a_{r+1,n}+\dots
+g_{1,n-1}^{(n,r)}\{A\}a_{n-1,n}.\] This
equality determines $a_{rn}$ because only $a_{rn}$ is unknown and
$g_{1r}^{(n,r)}\{A\}\ne 0$ by \eqref{cws}.

Suppose now that $r=1$.
Write $C$ in the form $D(A-\lambda_nI_n)$, in which
\[
 D=[d_{ij}]:=(A-\lambda_2I_n)
 (A-\lambda_3I_n)\cdots
 (A-\lambda_{n-1}I_n).
\]
Then
\[
c_n=d_{11}a_{1n}+d_{12}a_{2n}+\dots+ d_{1,n-1}a_{n-1,n}.
\]
This equality determines $a_{1n}$ since
only $a_{1n}$ is unknown and
\[
d_{11}=(\lambda _1-\lambda _2)(\lambda _1-\lambda _3)\cdots(\lambda _1-\lambda _{n-1})\ne 0.
\]

Therefore, we have determined all entries of $A$.
\end{proof}

\begin{proof}[{Proof of Theorem
\ref{thms}}] \label{kay}
By Lemmas \ref{oiy}(a) and \ref{ajs2}, each matrix $A\in M_n(f_n)$ is unitarily similar to a
matrix $A'\in M_n(f_n)$ of the form \eqref{jys8} via a diagonal unitary matrix.
Then $\|h(A_k)\|=\|h(A'_k)\|$ for all
$h\in\mathbb C[x]$ and
$k=1,\dots,n$. Thus, it suffices to prove Theorem
\ref{thms} for matrices $A,B\in M_n(f_n)$ of the form  \eqref{jys8}.

``$\Rightarrow$'' Let $A,B\in M_n(f_n)$ of the form  \eqref{jys8} be unitarily similar and have the same main diagonal. By Lemma \ref{oiy}(b), $A=B$ and so \eqref{blh} holds.

``$\Leftarrow$'' Let $A,B\in M_n(f_n)$ of the form  \eqref{jys8} satisfy \eqref{blh}. By Lemma \ref{ajs1}, $A=B$ since their indexed families  $\{\|h({A}_k)\|\}_{(h,k)}$ and
$\{\|h({B}_k)\|\}_{(h,k)}$ coincide.
\end{proof}

\end{document}